\documentclass{birkjour}

\theoremstyle{theorem}
\newtheorem{defn}{Definition}[subsection]
\newtheorem{thm}[defn]{Theorem}
\newtheorem{exmp}[defn]{Example}
\newtheorem{lem}[defn]{Lemma}
\newtheorem{rmk}[defn]{Remark}
\newtheorem{prop}[defn]{Proposition}
\newtheorem{cor}[defn]{Corollary}

\begin{document}
\title[A survey on composition operators]{A survey on composition operators on some function spaces}

\author[E. D'Aniello]{Emma D'Aniello}
\address{Dipartimento di Matematica e Fisica\\ 
Universit\`a degli Studi della Campania ``Luigi Vanvitelli"\\
Viale Lincoln n.5 \\
81100 Caserta\\
ITALIA}
\email{emma.daniello@unicampania.it}

\author[M. Maiuriello]{Martina Maiuriello}
\address{Dipartimento di Matematica e Fisica\\ 
Universit\`a degli Studi della Campania ``Luigi Vanvitelli"\\
Viale Lincoln n.5 \\
81100 Caserta\\
ITALIA}
\email{martina.maiuriello@unicampania.it}

\begin{abstract} 
We investigate some types of composition operators, linear and not, and conditions for some spaces to be mapped into themselves 
and for the operators to satisfy some good properties.   
\end{abstract}

\date{\today}


\maketitle

\section{Introduction}

The following non-linear operators, $C_ f$,  $x \mapsto f \circ x$ and $S_{h}$, $x(\cdot) \mapsto h(\cdot, x(\cdot) )$ called, respectively, the (autonomous) composition operator 
and the (non-autonomous) superposition operator, have been widely studied.  They especially appear in the process of solving certain non-linear integral equations. For instance,  
in \cite{ABM} and \cite{ABK}, the authors show that existence and uniqueness results for solutions of non-linear integral equations of Hammerstein-Volterra type
\begin{eqnarray*} 
& &x(t)=g(t) + \lambda \int _{0}^{t} k(t,s)f(x(s))ds, (t \geq0), \\ 
& &x(t)=g(t) + \lambda \int _{0}^{t} k(t,s)h(s,x(s))ds, (t \geq0)
\end{eqnarray*} 
and of Abel-Volterra type
 \[x(t)=g(t)+\int_{0}^{t} \frac{k(t,s)f(x(s))}{\vert t-s \vert^\nu} ds , (0 \leq t \leq 1)\] are closely related to existence and uniqueness results for solutions of operator equations involving $C_f$ and $S_{h}$. 
 Also, for example, in \cite{SA}, it is proved, for the integral equation of Volterra type in the Henstock setting, that the existence of a continuous solution depends, among other conditions, on the property of 
 mapping continuous functions into Henstock-integrable functions, satisfied by the involved non-autonomous superposition operator; in \cite{DMS},  
the authors  provide,  in the Henstock-Kurzweil-Pettis setting,  existence and closure results for integral problems driven by regulated functions, both in single- and set-valued cases (\cite{CDMS2}).
Hence, in many fields of non-linear analysis and its applications (in particular to integral equations), the following problem becomes of interest:

\hfill \break
 {\it Given a class $X$ of functions, find conditions on (and eventually characterise) the functions under which the generated operators map the space $X$ into itself.}
\hfill \break

\noindent
The case of the operator $S_{h}$ is called in the literature {\it Superposition Operator Problem} (\cite{AZ1}, \cite{BLS1}, \cite{BLS2}) or, sometimes, 
{\it Composition Operator Problem} (\cite{ABM}, \cite{AGM}) since it is also considered for the autonomous case $C_f$ and, in this simpler form as well, 
it is sometimes unexpectedly difficult. In addition to the action spaces of the non-linear operators $C_f$ and $S_h$, boundedness and continuity 
are properties which have also been the object of several studies: many results analysing such properties for composition operators on function spaces, among 
which $Lip$, $Lip_{\gamma}$, $BV$, $BV_{p}$, $AC$ and $W^{1,p},$ appeared in the last decades  (see, for instance, the papers cited throughout this note). 

\indent
This note is intended  to serve as a survey on the state of the art of some aspects and to describe some further properties of the non-linear  operators $S_{h}$ and 
$C_{f}$ ({\it left composition operator}), and the linear operator $T_{f}: x \mapsto x \circ f$ ({\it right composition operator}), discuss them and give examples. 
Clearly, the theory is wide and far from being complete. \\
This note is organised into four sections, including the introduction.\\
In Section 2, we briefly introduce the investigated function spaces, and we recall some main properties. \\
In Section 3, we analyse the non-linear operators $C_{f}$ and $S_{h}$. First, we investigate them on Lipschitz spaces and some spaces of functions of bounded variation, 
providing the main results in the literature with examples.  Then, we focus on spaces of Baire functions. In particular, we show that when the operator 
$C_{f}$ maps the space of Baire functions into itself, then it is automatically continuous.  We also characterise the non-linear operator 
$S_h$ which transforms  Baire one functions into maps of the same type, and we show how to construct a function $h$ easily which is not even Baire one 
but such that the associated operator $S_{h}$ maps the space of Baire functions into itself. \\
Section 4 is devoted to the linear composition operator $T_{f}$. We start by investigating Lipschitz spaces and some spaces of functions of bounded variation. 
In particular, our study shows that, unlike the case of left compositors, not all the investigated spaces have the same type of right compositors.  
Then, we study the linear operator $T_f$ on the space of Baire one functions and we develop some parallel results on the space of Baire two functions. 
Unlike the case of left Baire compositors which are the same for Baire classes of any order, and in particular for Baire one functions and Baire two functions, 
we show that it is not the case when we consider right composition. Namely, we show that right Baire one compositors do not coincide with right Baire two compositors.

\section{Preliminary definitions}
\indent 
In this section, we collect some basic notations, definitions and results, which will be needed in the sequel.

 \noindent
By $Lip([a,b])$ and $Lip_{\gamma}([a,b])$, we denote, respectively,  the space of all Lipschitz functions on $[a,b]$, and the space 
of all $\gamma$-Lipschitz (or H\"{older} continuous) functions on $[a,b]$, endowed with the usual norms 
\[ {\Vert f \Vert}_{Lip} = f(a) + Lip(f) \, \text{ and } \, {\Vert f \Vert}_{Lip_{\gamma}} = f(a) + Lip_{\gamma}(f),\]
with 
\[Lip(f) = \sup_{\substack{ x,y \in [a,b] \\ x \not= y}} \frac{\vert f(x) - f(y) \vert}{\vert x - y \vert} \, \text{ and } \, Lip_{\gamma}(f) = 
\sup_{\substack{ x,y \in [a,b] \\ x \not= y}} \frac{\vert f(x) - f(y) \vert}{{\vert x - y \vert}^{\gamma}}.\]

\subsection{$p$-variation, Jordan variation, Riesz variation}

\begin{defn}
Let $f$ be a real valued function defined on $H \subseteq \Bbb R$.  For $p > 0$ we denote by $V_{p}(f,H)$ the $p$-variation of $f$ on $H$, 
that is the least upper bound of the 
sums \[\sum_{i=1}^{n} \vert f(b_i)-f(a_i) \vert ^p,\] where $ \{ [a_i , b_i]\}_{i=1,..,n}$ is an arbitrary finite system of -overlapping intervals 
with $a_i, b_i \in H$, $i=1,...,n.$
\end{defn}
\noindent
If $H$ has a minimal as well as a maximal element, then $V_{p}(f,H)$ is the supremum of the sums  \[\sum_{i=1}^{n} \vert f(t_i)-f(t_{i-1}) \vert ^p,\] 
where 
$\min (H)=t_0<t_1<\dots < t_n =\max (H)$ and $t_i \in H$, $i=0,...,n.$ \\
\noindent

\noindent
From now on in this paragraph, we consider $f$ as a function defined on a closed interval of the real line, that is $f:[a,b] \rightarrow \Bbb R$.

\begin{defn}
We define $BV_{p}([a,b])=\{ f:[a,b] \rightarrow {\Bbb R} : V_{p}(f,[a,b])< + \infty \}$, i.e. $BV_{p}([a,b])$ is the space of functions of 
$p$-bounded variation on $[a,b]$.
\end{defn} 

\begin{defn}
When $p =1$, the variation $V_{1}(f,[a,b])$ is the Jordan variation, $V(f, [a,b])$, of $f$ on $[a,b]$. In particular, the space 
$BV_{1}([a,b])=\{ f:[a,b] \rightarrow {\Bbb R} : V(f,[a,b])< + \infty \}$ is the space of functions of bounded  Jordan variation on $[a,b]$ 
and it is simply denoted by $BV([a,b])$.
\end{defn}

\begin{rmk}
It is well-known  (\cite{ABM}, \cite{ABK}, \cite{AZ}) that the space $BV_{p}([a,b])$, $p \geq 1$, endowed with the norm 
$\Vert f\Vert _{BV_{p}}=\vert f(a)\vert +V_{p}(f,[a,b])^{\frac{1}{p}}$ 
is a Banach space. In particular, the space $BV([a,b])$ endowed with the norm $\Vert f\Vert _{BV}=\vert f(a)\vert +V(f,[a,b])$ 
is a Banach space.
Moreover, for $1 < p < q < \infty$, the following (strict) inclusions hold
\[BV([a,b]) \subset BV_{p}([a,b]) \subset BV_{q}([a,b]) \subset B([a,b])\]
where $B([a,b])=\{f:  [a,b] \rightarrow {\Bbb R}; f \text{ is bounded}\}$.
\end{rmk}

As it is well-known, the space $BV([a,b])$ is not closed under composition. For example, take $[a,b] =[0,1]$ and 
$f = g \circ h$, where $g(x) = \sqrt{x}$ and $h(x)$ is defined as
\[h(x) = \left\{ \begin{array}{ll}
0  & \mbox{if } x =0\\ 
x^{2} {\sin}^{2}(\frac{1}{x}) & otherwise.\\
\end{array}
\right.\]

In the case of continuous functions, we have the following definition.

\begin{defn}
Let $f$ be continuous on $[a,b]$. Let $G$ be the union of all open subintervals of $(a,b)$ on which $f$ is either strictly monotonic or constant. 
The set of points of varying monotonicity of $f$ is 
defined as \[K_f=[a,b]\setminus G .\]
\end{defn}

\begin{thm} (\cite{LP}: Theorem 2.3) \label{THMLP}
For every $f \in C([a,b])$ and $p \geq 1$ we have \[ V_{p}(f,K_{f})=V_{p}(f,[a,b]).\]
\end{thm}

\noindent
Let $CBV_{p}([a,b]) = \{ f \in C([a,b]): V_{p}(f,K_{f})<\infty\}$.  

\begin{cor} (\cite{LP}: Corollary 2.4)
If $p \geq 1$ then $CBV_{p}([a,b])$ is the family of those $f \in C([a,b])$ for which $V_{p}(f,[a,b])<+\infty$, that is $CBV_p([a,b])=C([a,b]) \cap BV_p([a,b]).$
\end{cor}

\noindent
As remarked in \cite{LP}, no analogous statement to Theorem \ref{THMLP} holds if $0 < p < 1$ since the only continuous functions $f$ with $V_{p}(f,[a,b])<+\infty$ 
are constant. On the other hand, $CBV_{p}([a,b])$ contains, for example, the continuous, strictly monotone functions on $[a,b]$.\\

Now, we introduce another type of variation: the Riesz variation.

\begin{defn}
Let $\mathcal P$ be the family of all partitions of the interval $[a,b]$. Given a real number $p \geq 1$, a partition $P=\{ t_0,...t_m\}$ of $[a,b]$, 
and a function $f:[a,b] \rightarrow \Bbb R$, the non-negative real number \[ RV_p (f,P)=RV_p(f,P,[a,b])= \sum_{j=1}^{m} \dfrac{\vert f(t_j)-f(t_{j-1}) \vert ^p}{(t_j - t_{j-1})^{p-1}} \] 
is called the Riesz variation of $f$ on $[a,b]$ with respect to $P.$ The (possibly infinite) number \[ RV_p(f)=RV_p(f,[a,b])= \sup \{ RV_p(f,P,[a,b]) : P \in {\mathcal P} \}, \] 
where the supremum is taken over all the partitions of $[a,b]$, is called the total Riesz variation of $f$ on $[a,b]$. In case $RV_p(f) < \infty$ we say that $f$ has bounded 
Riesz variation (or bounded $p$-variation in Riesz' sense) on $[a,b]$, and we write $f \in RBV_p([a,b])$.
\end{defn}

\begin{rmk}
It is well-known (\cite{ABM}, \cite{ABK}) that the space $RBV_p([a,b])$ equipped with the norm $\Vert f\Vert _{RBV_p}= \vert f(a) \vert + RV_p (f)^{\frac{1}{p}}$ is 
a Banach space. 
\end{rmk}

By $AC([a,b])$, we denote the space of all absolutely continuous functions on $[a,b]$.
Moreover, $AC([a,b])$ is closed in $BV([a,b])$, and therefore it is a Banach space with respect to the $BV$-norm, which is equivalent to the  $W^{1,1}$-norm
\[\Vert f \Vert_{W^{1,1}} = {\Vert f \Vert}_{L^{1}} +   {\Vert f^{'}\Vert}_{L^{1}}.\]
In fact, it coincides with $W^{1,1}([a,b])$  (see, for instance, \cite{ABM}: Proposition 3.24;  \cite{ABK}: page 10, 1.1.16).

Recall the following useful (strict) inclusions.  
\begin{enumerate}
\item{For $1 < q < p < + \infty$,
\begin{align*}
Lip([a,b]) &  \subset    W^{1,p}([a,b]) \subset W^{1,q}([a,b]) \\
& \subset   AC([a,b]) \subset C([a,b]) \cap BV([a,b])\\
&  \subset  B([a,b]).\\
\end{align*}}
\item{Let $0 < \gamma < 1$. Then 
\[Lip_{\gamma}([a,b]) \not\subseteq BV([a,b]).\]
Moreover, there exists $f \in \cap_{0 < \gamma < 1} Lip_{\gamma}([a,b]) \setminus BV([a,b])$ (\cite{ABM}: Example 1.23 and Example 1.24).}
\end{enumerate}

\noindent
Functions of bounded Riesz variation are particularly interesting since they are related to 
Sobolev spaces: the space $RBV_p([a,b])$ is basically the same 
as the space $W^{1,p}([a,b])$, by the following well-known theorem. 

\begin{thm}[Riesz Theorem] (\cite{ABK}: Theorem 1.3.5) \label{theoRV} Let $1<p< \infty$. A function $f:[a,b] \rightarrow \Bbb R$ 
belongs to $RBV_p([a,b])$ if and only if $f \in AC([a,b])$ 
and $f' \in L^p([a,b]).$ Moreover, in this case the equality \[ RV_p(f)= \Vert f' \Vert ^p_{L^p([a,b])}= \int _{a} ^{b}  \vert f'(t) \vert ^p dt \] holds, 
where $RV_p (f)$ is the $p$-variation of $f$ in Riesz' sense.
\end{thm}

\noindent
For $p=1$, $RBV_1([a,b])=BV([a,b])$. Hence, the previous theorem does not hold for $p=1$ as a function in 
$BV([a,b])$ usually does not need to be continuous and therefore nor absolutely continuous.

\subsection{Baire functions}

\indent
Let $X$ be a Polish space, that is a  separable and completely metrizable space. 
Recall that an $F_{\sigma}$ set is a countable union of closed sets, a $G_{\delta}$ set is a countable intersection  of open sets, and a $G_{\delta \sigma}$ 
set is a countable union of $G_{\delta}$ sets (\cite{BBT}). In every metrizable space, any open set is an $F_{\sigma}$ set (\cite{AK}). \\
A real valued function $g: X \rightarrow {\Bbb R}$ is said to be {\it Baire one} if there exists a sequence $\{g_{k}\}_{k \in {\Bbb N}}$ of continuous functions 
$g_{k}: X \rightarrow  {\Bbb R}$ such that $\lim_{k \rightarrow + \infty} g_{k}(x) = g(x)$, for every $x \in X$.  These functions are so called since they were first defined and studied 
by Baire (\cite{BAI}). Clearly, each continuous function is of Baire class one. \\
In general, a real valued function $g: X \rightarrow {\Bbb R}$ is said to be of {\it Baire class $n$}, $n \in \Bbb N$, if there exists a sequence $\{g_{k}\}_{k \in {\Bbb N}}$ of functions of 
Baire class $n-1$, $g_{k}: X \rightarrow  {\Bbb R}$, such that $\lim_{k \rightarrow + \infty} g_{k}(x) = g(x)$, for every $x \in X$. \\
Denote by ${\mathcal B}_{0}(X)$ the collection of real valued continuous functions on $X$, that is  ${\mathcal B}_{0}(X) = C(X)$, 
and by ${\mathcal B}_{n}(X)$, $n \geq 1$, the collection of real valued Baire $n$ functions on $X$.  \\
Then, the following (strict) inclusions hold: 
\[C(X) = {\mathcal B}_{0}(X)  \subset {\mathcal B}_{1}(X) \subset \cdots \subset {\mathcal B}_{n}(X) \subset  {\mathcal B}_{n +1}(X) \subset \cdots.\] 
Several equivalent definitions of Baire class one functions have been obtained already: it is well-known that ``$g$ is Baire one if 
and only if for every open set $A$, $g^{-1}(A)$ is an $F_{\sigma}$ set", and that ``$g$ is Baire two if and only if for every open set $A$, $g^{-1}(A)$ is a $G_{\delta \sigma}$
set" (see, for instance, \cite{KE} and \cite{Z}).\\
Given $g: X \rightarrow {\Bbb R}$, the following are equivalent: 
\begin{enumerate}
\item{$g$ is Baire one;}
 \item{for every open subset $A$ of ${\Bbb R}$, $g^{-1}(A)$ is an $F_{\sigma}$ set;}
\item{for every closed set $C$ in $X$, the restriction $g_{|C}$ has a point of continuity in $C$.}
\end{enumerate} 
\noindent
Clearly, if a function $g: X \rightarrow {\Bbb R}$ has countably many discontinuity points then it is Baire one. In particular, if 
$g: [a, b]  \rightarrow {\Bbb R}$ is monotone, or of bounded variation, then $g$ is Baire one. In general, functions of Baire class one 
play an important role in applications. For example, semi-continuous functions and derived functions, all belong to this class (\cite{BBT}, \cite{KU}). 
Some interesting, very recent results concerning fixed points of Baire 
functions and the so called equi-Baire property can be found in \cite{AL0} and \cite{AL}.\\
If $g$ is Baire one, then the set of points of continuity of $g$ is a residual subset of $X$. This last property is not a characterisation 
as the following example shows (\cite{NAT}: page 148, Example IV).  

\begin{exmp}
Let $X=[0,1]$. Let $C$ be the Cantor ternary set. The set $C$ has Lebesgue measure zero and is of first category since it is nowhere dense. Let $C_{0}$ be the collection 
of the points of $P$ which are not endpoints of the complementary intervals. Let $f = {\chi}_{C}$ and $g= {\chi}_{C_{0}}$.  Then $f$ and $g$ are continuous at points of 
$[0, 1] \setminus C$ and discontinuous at points of $C$. But $f$ is Baire one as it is the characteristic function of a closed set but $g$ is not Baire one as $g_{|C}$ is 
discontinuous at every point. 
\end{exmp}
\noindent
Another well-known example of non Baire one functions is the Dirichlet function. 

\begin{exmp} 
Let $X= [0,1]$.  The Dirichlet function is the map $g(\cdot) = {\chi}_{{\Bbb Q} \cap [0,1]}(\cdot)$. List all the rationals in $[0,1]$ as $r_{1}$, $r_{2}$, $\dots$, $r_{k}$, $\dots$. 
Define, for each $n \in {\Bbb N}$, 
\[g_{n}(x) = \left\{ \begin{array}{ll}
1  & \mbox{if } x \in \{r_{1}, \cdots, r_{n}\} \\ 
0 & otherwise.\\
\end{array}
\right. \]
As $g_{n}$ has finitely many discontinuity points, it is of Baire class one. The Dirichlet map is the pointwise limit of the sequence $\{g_{n}\}_{n \in {\Bbb N}}$. 
So, it is Baire two but not Baire one.    
\end{exmp}

\noindent
The following is a beautiful, natural characterisation of a Baire one function. 
\begin{thm} (\cite{LTZ}: Theorem 1)  \label{THMLEC0}
Suppose $f: X \rightarrow Y$ is a mapping between complete separable metric spaces $(X, d_{X})$ and $(Y, d_{Y})$. Then the following statements are equivalent.
\begin{enumerate}
\item{For any $\epsilon >0$, there exists a positive function $\delta$ on $X$ such that 
\break
$d_{Y}(f(x), f(y) ) < \epsilon$ whenever $d_{X} (x,y) < \min \{\delta(x), \delta(y)\}$.}
\item{The function $f$ is of Baire class one.}
\end{enumerate}
\end{thm}

\begin{rmk} \label{RMKLEC}
The function $\delta$ of Theorem  \ref{THMLEC0} can be chosen to be Baire one as shown in Corollary 33 of \cite{LEC}.
\end{rmk}

In the sequel, sometimes, when understood, in the above mentioned spaces, we omit $X$ (we write, for example, ${\mathcal B}_{1}$ rather than ${\mathcal B_{1}(X)}$). 

\section{Two types of non-linear operators: (left) composition operators and  superposition operators}

\indent
Recall that an operator between two normed spaces is said to be {\it  bounded} if it maps bounded sets into bounded sets. 
Clearly, unlike the case of linear operators, in the non-linear case, the two properties of being  bounded and being continuous are not equivalent. 
They are not even, in general, related:  a non-linear operator may be continuous without being bounded, 
or bounded without being continuous. 

\begin{defn}
Let $J$ be an arbitrary interval. Let $f:{\Bbb R} \rightarrow \Bbb R$. The operator \[g \mapsto f \circ g \] where $g:J \rightarrow \Bbb R$ is an arbitrary function, is called 
the {\it (autonomous) composition operator generated by the function $f$}. It is usually denoted by $C_f $. Hence, for each $g:J \rightarrow \Bbb R$, 
$C_{f}(g): J \rightarrow {\Bbb R}$, is defined as $C_{f}(g)(\cdot) = f(g(\cdot))$.
\end{defn}

\begin{defn}
Let $J$ be an arbitrary interval. Let $h:J \times {\Bbb R} \rightarrow \Bbb R$. The operator 
\[g (\cdot) \mapsto h (\cdot,  g( \cdot)), \] 
where $g:J \rightarrow \Bbb R$ is an arbitrary function, is called the {\it (non-autonomous) superposition operator generated 
by the function $h$}. It is usually denoted by $S_h $. 
Hence, for each $g:J \rightarrow \Bbb R$, $S_{h}(g): J \rightarrow {\Bbb R}$ is defined as $S_{h}(g)(\cdot) = h(\cdot, g(\cdot))$.
\end{defn}

\subsection{Composition Operators}

\paragraph{Composition operators on Lipschitz functions and some spaces of functions of bounded variation}

\begin{thm}(\cite{ABM}: Theorem 5.9)
The operator $C_f$ maps the space $BV([a,b])$ into itself if and only if the corresponding function $f$ is locally Lipschitz on $\Bbb R$, 
i.e. for each $r>0$, there exists $k(r)>0$ such that
 \[ (\star) \hspace{3cm} \vert f(u)-f(v) \vert \leq k(r) \vert u-v\vert, \hspace{0,3 cm} (u,v \in {\Bbb R}, \vert u \vert , \vert v \vert \leq r).\]
\end{thm} 
 
\begin{thm} (\cite{ABK}: Theorem 3.4.1; \cite{ABM}: Theorem 5.24) \label{thmautBV}
Let $1 < p < \infty$, $0 < \gamma \leq 1$. The following conditions are equivalent.
\begin{enumerate}
\item[(a)]{The function $f: {\Bbb R}\rightarrow {\Bbb R}$ satisfies the local Lipschitz condition $(\star)$.}
\item[(b)]{The operator $C_f$ maps the space $BV_{p}([a,b])$ into itself.}
\item[(c)]{The operator $C_f$ maps the space $BV([a,b])$ into itself.}
\item[(d)]{The operator $C_f$ maps the space $AC([a,b])$ into itself.}
\item[(e)]{The operator $C_f$ maps the space $RBV_{p}([a,b])$ into itself.}
\item[(f)]{The operator $C_f$ maps the space $Lip_{\gamma}([a,b])$ into itself.}
\end{enumerate}
Moreover, in this case the operator $C_f$ is automatically bounded.
\end{thm}
 
\begin{thm} (\cite{ABK}: Theorem 3.1.7)  Under the hypothesis $(\star)$, the operator $C_f$ is automatically continuous in $BV([a,b])$.
\end{thm}

\begin{thm} (\cite{ABK}: Theorem 3.4.2)  Under the hypothesis $(\star)$, the operator $C_f$ is automatically continuous in $RBV_{p}([a,b])$, $1 < p < \infty$.
\end{thm}
\noindent
Note that the equivalence between conditions (a) and (d) of Theorem \ref{thmautBV} is a particular case of the following more general result involving Sobolev spaces, 
which follows from Theorem 1 in \cite{MM}:
\begin{thm}
Let $1 \leq q \leq p < \infty$, and let $f: {\Bbb R}\rightarrow {\Bbb R}$ be a Borel function. Then the composition operator $C_f$ maps the space $W^{1,p}([a,b])$ 
into $W^{1,q}([a,b])$ if and only if $f$ satisfies the local Lipschitz condition $(\star)$. Moreover, the operator $C_f$ is bounded and the following inequality holds: 
\[ \Vert C_f (g) \Vert _{W^{1,q}} \leq b(M)(1+\Vert g \Vert_{W^{1,p}}),\] where $\Vert g \Vert_{W^{1,p}} \leq M$ and $b(M)$ is a constant depending on $M$.
\end{thm}

Some spaces behave well with respect to the composition operator, as the following results show: 

\begin{thm}  (\cite{ABM}: Theorem 5.20) The operator $C_{f}$ maps $C([a,b])$ into itself if and only if $f$ is continuous on ${\Bbb R}$.  In this case, the operator 
$C_{f}$ is automatically bounded and continuous in the norm $\Vert \cdot \Vert_{C}$. 
\end{thm}

\begin{rmk} Let $0 < \gamma \leq 1$. As example  5.25 in \cite{ABM} shows, there exists a composition operator $C_{f}$ that maps $Lip_{\gamma}([0,1])$ into 
itself but is not continuous. In order to have the continuity of $C_{f}$, extra properties have to be satisfied by the generating function $f$. In \cite{GS} the authors 
prove that $C_{f}$ is continuous on $Lip_{\gamma}([a,b])$ if and only if $f \in C^{1}({\Bbb R})$.  In Theorem 5.26 of \cite{ABM}, the authors prove that the 
continuity of $C_{f}$, defined from $Lip_{\gamma}([a,b])$ into itself, is equivalent to its uniform continuity on bounded subsets.
 \end{rmk}

Given a space $X$ of real functions defined on a real interval $J$, in accordance with the terminology used in \cite{Z} by Zhao in the particular case of Baire functions, 
we say that a function $f:{\Bbb R} \rightarrow {\Bbb R}$ is a {\it left $X$ compositor} if $f \circ g$ belongs to $X$ whenever $g$ is an element 
of  $X$. Hence, we can re-write Theorem \ref{thmautBV} as a characterisation of left  compositors for some spaces. 

\begin{thm} \label{theoautBV-1}
Let $1 < p < \infty$, $0 < \gamma \leq 1$. Let $f: {\Bbb R} \rightarrow {\Bbb R}$. 
The following statements are equivalent.
\begin{enumerate}
\item[(a)]{The function $f$ satisfies the local Lipschitz condition $(\star)$.}
\item[(b)]{The  function $f$ is a left $BV_{p}([a,b])$ compositor.}
\item[(c)]{The  function $f$ is a left $BV([a,b])$ compositor.}
\item[(d)]{The function $f$ is a left $AC([a,b])$ compositor.}
\item[(e)]{The  function $f$ is a left $RBV_{p}([a,b])$ compositor.}
\item[(f)]{The  function $f$ is a left $Lip_{\gamma}([a,b])$ compositor.}
\end{enumerate}
\end{thm}

\begin{rmk}
Hence, the collections of left $AC([a,b])$, $BV_{p}([a,b])$ ($1 \leq p < \infty$),  $RBV_{p}([a,b])$  ($1 <p < \infty$), $Lip_{\gamma}([a,b])$  ($0< \gamma \leq 1$)  
compositors are all the same, namely they all coincide with the collection of all maps satisfying $(\star)$.
\end{rmk}
   
\noindent    
From Theorem \ref{theoautBV-1} and the fact that the composition, the sum and the product of two functions satisfying the local Lipschitz condition $(\star)$ 
still satisfy the local Lipschitz condition $(\star)$, we have the following proposition. 
 
\begin{prop} Let $1 < p < \infty$, $0 < \gamma \leq 1$.  Let $X= BV_{p}([a,b])$, $BV([a,b])$, $AC([a,b])$, $RBV_p([a,b])$, $Lip_{\gamma}([a,b])$. Then, the following hold.  
\begin{enumerate}
 \item{If $f: {\Bbb R} \rightarrow {\Bbb R}$ and $g: {\Bbb R} \rightarrow {\Bbb R}$ are left $X$ compositors then so is the sum $f+g$.}
\item{If $f: {\Bbb R} \rightarrow {\Bbb R}$ and $g: {\Bbb R} \rightarrow {\Bbb R}$ are left $X$ compositors then so is the product $fg$.}
\item{If $f: {\Bbb R} \rightarrow {\Bbb R}$ and $g: {\Bbb R} \rightarrow {\Bbb R}$ are left $X$ compositors then so is the composition $f \circ g$.}
\end{enumerate}
In particular, for these spaces, the collection of left compositors is a vector space and an algebra. 
\end{prop}

\paragraph{Composition operators on spaces of Baire functions}
\hfill \break
Let $g: {\Bbb R}  \rightarrow {\Bbb R}$ and $f: {\Bbb R} \rightarrow{\Bbb R}$. If $g$ is a Baire one function and $f$ is continuous, then 
the composition function $f \circ g$ is Baire one but, as it is well-known,  the composition of two Baire one functions is not necessarily 
Baire one. Here is a well-known example.

\begin{exmp} (\cite{Z}: Example 1) \label{EXMPZ}
Let $f: {\Bbb R} \rightarrow {\Bbb R}$ be defined as
 \[f(x) = \left\{ \begin{array}{ll}
1  & \mbox{if } x= \frac{1}{n}, n \in {\Bbb N}\\
 0 & {otherwise} \\
\end{array}
\right.\] 
and $g: {\Bbb R} \rightarrow {\Bbb R}$ be the Riemann function defined as
 \[g(x) = \left\{ \begin{array}{lll}
\frac{1}{q}  & \mbox{if } x= \frac{p}{q}, p \text{ and }q  \text{ are co-prime integers and } 0< q\\
 1 & \mbox{if }  x=0\\
 0 & {otherwise.} \\
\end{array}
\right.\] 
 Then $f \circ g$ is the Dirichlet function, that is not Baire one.  \\
 Notice that, by taking $f= \chi_{(0,1]}$ and $g$ the same as above, we still have that $f \circ g$ is the Dirichlet function. This also shows, as $f$ has a finite number of discontinuity 
 points (namely, exactly one: $x=0$), that the last claim in \cite{Z} is not true. 
 \end{exmp}  

\noindent
Thus, as done above for other spaces,  it is natural to ask which functions $f$ have the property that their composition with any Baire one function is still of Baire 
class one, that is $C_{f}$ maps the space of Baire one functions in itself.  

\begin{rmk} As already mentioned, Zhao, in \cite{Z}, calls a function $f: {\Bbb R} \rightarrow {\Bbb R}$ for which $C_{f}({\mathcal B}_{1}) \subseteq {\mathcal B}_{1}$ a 
{\it left Baire one compositor}. That is, $f$ is a left Baire one compositor if and only if $f \circ g$ is Baire one whenever $g$ is a Baire one function. 
\end{rmk}

\noindent
The following result follows from Theorem 3 of \cite{FC1}, in the case of Baire one functions and, more generally, from Theorem D of \cite{KM}, for Baire functions of class $n$.  

\begin{thm} \label{thmBairen}
Let $n \in {\Bbb N}$. Let $f: {\Bbb R} \rightarrow {\Bbb R}$. The following are equivalent:
\begin{enumerate}
\item[(a)]{The function $f$ is continuous.}
\item[(b)]{The operator $C_f$ maps the space ${\mathcal B}_{n}$ into itself.}
\end{enumerate}
\end{thm}

\noindent
We have that, in the case of Baire functions, the composition operator behaves well.  Namely, the following holds. 

\begin{thm} Let $n \in {\Bbb N}$. Let $f: {\Bbb R} \rightarrow {\Bbb R}$. 
If $C_f$ maps the space ${\mathcal B}_{n}$ into itself then it is automatically continuous with respect to pointwise convergence. 
 \end{thm}
\begin{proof}
As $C_f$ maps the space ${\mathcal B}_{n}$  into itself, by Theorem \ref{thmBairen}, $f$ is continuous. Assume that $\{g_{k}\}_{k \in \Bbb N}$ is a 
sequence in  ${\mathcal B}_{n}$ converging pointwise to a function $g$ of Baire class $n$.  Then $C_{f}(g_{k})=  f \circ g_{k}$ is a sequence of Baire 
functions of class $n$ pointwise converging to the Baire class $n$ function $C_{f}(g) = f \circ g$. \\
\end{proof}

Recall that a subset $A$ of ${\mathcal B}_{n}$ is said to be {\it bounded} if each element $h$ in $A$ is bounded, 
that is ${\Vert h \Vert}_{\infty} = \sup \vert h(x) \vert < \infty$, 
and, moreover,  there exists $M>0$ with ${\Vert h \Vert}_{\infty} \leq M, \forall h \in A$. 
When $f$ is continuous, it is straightforward that the operator $C_f$ is locally bounded on the space $B(\Bbb R)$ of
bounded functions on the reals with the sup-norm, and hence, on ${\mathcal B}_{n}$ as well. 
(Let $A$ be a bounded set of functions and let $M >0$ be such that ${\Vert h \Vert}_{\infty} \leq M$, $\forall h \in A$. 
As $f$ is continuous on ${\Bbb R}$, the restriction of $f$ to the compact  $[-M, M]$ admits a maximum. Call this maximum $L$. 
Then, for  each $h \in A$, we have ${\Vert C_{f}(h) \Vert}_{\infty} =   {\Vert f \circ h \Vert}_{\infty} \leq L$. 
Hence, $C_{f}(A)$ is bounded.) \\

{\bf A natural question arises:} what about right compositors in all the previous spaces? Clearly, right composition $g \mapsto g \circ f$, defined with 
a suitable $f$ and on suitable spaces, is linear. \\
This question is investigated in Section 4.

\subsection{Superposition operators}

\paragraph{Superposition operators on Lipschitz functions and some spaces of functions of bounded variation}
\hfill \break
\indent
As for the case of composition operators, we are interested in investigating which spaces are mapped by $S_{h}$ into themselves 
and, in general, in the properties of $S_{h}.$
Natural sufficient conditions are the following: 

\begin{prop} (\cite{AZ1}: Proposition 3.1) Let $h:[a, b] \times {\Bbb R} \rightarrow {\Bbb R}$, $0< \gamma <1$, and $1< p < + \infty$.
\begin{enumerate}
\item{If $h(\cdot, y) \in Lip$ uniformly w.r.t. $y$ then $S_{h}$ maps $Lip$ into itself.}
\item{If $h(\cdot, y) \in BV$ uniformly w.r.t. $y$ and  $h(x, \cdot) \in Lip$  uniformly w.r.t. $x$ then $S_{h}$ maps $BV$ into itself.}
\item{If $h(\cdot, y) \in AC$ uniformly w.r.t. $y$ and  $h(x, \cdot)\in Lip$  uniformly w.r.t. $x$ then $S_{h}$ maps $AC$ into itself.}
\item{If $h(\cdot, y) \in Lip_{\gamma}$ uniformly w.r.t. $y$ and $h(x, \cdot) \in Lip$ uniformly w.r.t. $x$ then $S_{h}$ maps $Lip_{\gamma}$ into itself.}
\item{If $h(\cdot, y) \in BV_{p}$ uniformly w.r.t. $y$ and  $h(x, \cdot) \in Lip$ uniformly w.r.t. $x$ then $S_{h}$ maps $BV_{p}$ into itself.}
\end{enumerate}
\end{prop}
\noindent
Other known results are the following.

\begin{thm} (\cite{ABM}: Theorem 6.1) Let $h :[a, b] \times {\Bbb R} \rightarrow {\Bbb R}$. The operator $S_{h}$ maps $C([a,b])$ into itself if and 
only if $h$ is continuous on $[a,b] \times \Bbb R$. In this case, 
the operator $S_{h}$ is automatically bounded and continuous in the norm $\Vert \cdot \Vert_{\infty}$.
\end{thm}

\begin{thm} (\cite{ABM}: Theorem 6.4) Let $0 < \gamma \leq 1$. Let $h :[a, b] \times {\Bbb R} \rightarrow {\Bbb R}$. The operator $S_{h}$ maps 
$Lip_{\gamma}([a,b])$ into itself and is bounded  w.r.t. the norm $\Vert \cdot \Vert_{Lip_{\gamma}}$ if and only if $h$ satisfies the mixed local H\"{o}lder-Lipschitz condition 
\[\vert h(s,u) - h(t,v) \vert  \leq k(r) ({\vert s-t \vert}^{\gamma} + \vert u - v \vert) \hspace{1.5cm} (a \leq s,t \leq b, \vert u \vert, \vert v \vert \leq r).\]  
In particular, the function $h$ is then necessarily continuous on $[a,b] \times {\Bbb R}$. 
\end{thm}
 
\begin{thm} (\cite{BBKM}: Theorem 3.8) \label{THME}
Suppose that $h:[a,b] \times {\Bbb R} \rightarrow \Bbb R$ is a given function. The following conditions are equivalent:
\begin{itemize}
\item[(i)]{the non-autonomous superposition operator $S_{h}$ maps the space $BV([a,b])$ into itself and is locally bounded;}
\item[(ii)]{for every $r>0$ there exists a constant $M_r >0$ such that for every $k \in \Bbb N$, every finite partition $a=t_0< \cdots < t_k=b$ 
of the interval $[a,b]$ and every finite sequence $u_0,u_1,..., u_k \in [-r,r] $ with $\sum_{i=1}^k \vert u_i - u_{i-1} \vert \leq r,$ the following inequalities 
hold \[ \sum_{i=1}^k \vert h(t_i,u_i) -h(t_{i-1}, u_{i}) \vert \leq M_r \text{ and } \sum_{i=1}^k \vert h(t_{i-1},u_i) -h(t_{i-1}, u_{i-1}) \vert \leq M_r .\]}
\end{itemize}
\end{thm}

In \cite{M}, the author presents necessary and sufficient conditions for the continuity of a non-autonomous superposition operator in the $BV([a,b])$ case:

\begin{thm} (\cite{M}: Theorem 10)
Suppose that $h:[a,b] \times {\Bbb R} \rightarrow \Bbb R$ is a function such that the superposition operator $S_{h}$ maps the space $BV([a,b])$ into itself. 
Let $x \in BV([a,b])$ be fixed. The following conditions are equivalent:
\begin{itemize}
\item[(i)]{the superposition operator $S_{h}$ is continuous at $x$;}
\item[(ii)]{for each $t \in [a,b]$, the function $u \in {\Bbb R}\mapsto h(t,u)-h(t,x(t))$ is continuous at $u = x(t)$ and for every $\epsilon > 0$ there exists $\delta > 0$ 
such that, for every $k \in \Bbb N$, every partition $a=t_0 <...< t_k =b$ of the interval $[a,b]$, and every finite sequence $u_0, u_1,...,u_k\in  [-\delta, \delta]$ 
with $\sum _{i=1}^k \vert u_i - u_{i-1}\vert  \leq \delta$, we have \[ \sum_{i=1}^k \vert [h(t_i ,u_i +x_i)-h(t_{i-1},u_i+x_{i-1})]-[h(t_i,x_i)-h(t_{i-1},x_{i-1})]  \vert \leq \epsilon , \] 
and  
\[ \sum_{i=1}^k \vert {h(t_{i-1} ,u_{i} +x_{i-1})-h(t_{i-1},u_{i-1}+x_{i-1})}  \vert \leq \epsilon , \] 
where $x_i = x(t_i), i\in \{0,...,k\}$.}
\end{itemize}
\end{thm}

The following result is a special case of Theorem 3.8 in \cite{BBKM} (when $X=BV_{\varphi}$ with the Young function $\varphi(t) = t^{p}$). 

\begin{thm} 
Let $p \geq 1$. Suppose that $h:[a,b] \times {\Bbb R} \rightarrow \Bbb R$ is a given function. Consider the following conditions:
\begin{itemize}
\item[({\it a})]{for every $r>0$ there exists a constant $M_r >0$ such that for every $k \in \Bbb N$, every finite partition $a=t_0< \cdots < t_k=b$ 
of the interval $[a,b]$ and every finite sequence $u_0,u_1,..., u_k \in [-r,r] $ with $\sum_{i=1}^k {\vert u_i - u_{i-1} \vert}^{p} \leq r$ the following inequalities hold 
\[ \sum_{i=1}^k {\vert h(t_i,u_i) -h(t_{i-1}, u_{i}) \vert}^{p} \leq M_r \text{ and } \sum_{i=1}^{k} {\vert h(t_{i-1},u_i) - h(t_{i-1}, u_{i-1}) \vert}^{p} \leq M_r.\]}
\item[({\it b})]{the non-autonomous superposition operator $S_{h}$ maps the space $BV_{p}([a,b])$ into itself.}
\end{itemize}
Then, $(a)$ implies $(b)$. Moreover, $S_{h}$ is locally bounded. 
\end{thm}

In \cite{BBKM}, the authors provide the following interesting example of an operator $S_{h}$ mapping $BV$ 
into itself without being bounded or continuous. They take 
\[h(t, u) = \left\{ \begin{array}{ll}
\frac{1}{u}  & \mbox{if } u \not=0\\ 
0 & otherwise.\\
\end{array}
\right.\]

As far as we know, up to now, no characterisation of the functions $h$ is known for the associated operator $S_{h}$ 
to map $BV_{p}([a,b])$ into itself, with $p >1$. \\

\indent
Other interesting results concerning the operator $S_h$ on the spaces mentioned above and on other spaces like, for instance,  Sobolev spaces 
and Besov spaces, also in higher dimensions, can be found, for example, in \cite{AZ1}, \cite{BLS1}, \cite{BLS2}. \\

\paragraph{Superposition operators on some spaces of Baire functions}

\begin{thm} Let $h: [0,1] \times {\Bbb R} \rightarrow {\Bbb R}$.
Then, the following statements are equivalent.
\begin{enumerate}
\item{For every positive function $g \in {\mathcal B}_{1}([0,1])$, for every positive $\epsilon$,  there exists a positive function 
$\delta$ on $[0,1]$, of Baire class one, such that $\vert t - s \vert < \min \{\delta(t), \delta(s)\}$ implies 
$\vert h(t, g(t)) - h (s, g(s)) \vert <  \epsilon$. }
\item{For every $g \in {\mathcal B}_{1}([0,1])$, the function $h(\cdot, g(\cdot)) \in {\mathcal B}_{1}([0,1])$.}
\end{enumerate}
\end{thm}
\begin{proof} The proof follows from Theorem \ref{THMLEC0} and Remark \ref{RMKLEC} applied, for any given $g \in {\mathcal B}_{1}([0,1])$, 
to the map, clearly depending on $g$, $f(x) = h(x, g(x))$. More precisely, for any $g \in  {\mathcal B}_{1}([0,1])$, the following are equivalent:
\begin{enumerate}
\item{for every positive $\epsilon$,  there exists a positive Baire one function $\delta$ on $[0,1]$ such that $\vert t - s \vert < \min \{\delta(t), \delta(s)\}$ implies 
$\vert h(t, g(t)) - h (s, g(s)) \vert <  \epsilon$; }
\item{the function $h(\cdot, g(\cdot)) \in {\mathcal B}_{1}([0,1])$.}
\end{enumerate}
\end{proof}

\begin{prop} \label{PropBaire}
If $h: [0,1] \times {\Bbb R} \rightarrow {\Bbb R}$ is continuous then, for every $g \in {\mathcal B}_{1}([0,1])$, the function 
$h(\cdot, g(\cdot))$ is  in ${\mathcal B}_{1}([0,1])$. Hence, the non-autonomous superposition operator $S_{h}$ maps the space ${\mathcal B}_{1}([0,1])$ into itself. 
\end{prop}

\begin{proof}
This follows from the fact that the composition of a continuous function with a Baire one function is a Baire one function.  
\end{proof}

\begin{rmk}
Proposition \ref{PropBaire} cannot be reverted. A function $ h: [0, 1] \times \nobreak {\Bbb R} \rightarrow \nobreak {\Bbb R}$ need not be of Baire class one to generate a 
superposition operator $S_{h}$ in ${\mathcal B}_{1}$. For example, the function $h: = \chi_{\{0\} \times {\Bbb Q}}$ has the property that 
$h(t, g(t)) = \chi_{\{0\}}(t)$  if $g(0) \in {\Bbb Q}$, and $h(t,g(t)) \equiv 0$ if $g(0) \notin {\Bbb Q}$, therefore $S_{h}$ maps ${\mathcal B}_{1}$ into itself. 
Since the restriction $h(0, \cdot)$ is a Dirichlet function, $h$ cannot be of Baire class one, let alone continuous. 
\end{rmk}

\section{A type of linear operators: right composition operators}

\begin{defn}
Let $I$ and $J$ be compact intervals. Let $f:J \rightarrow I$. The operator \[g \mapsto g \circ f ,\] where $g: I  \rightarrow \Bbb R$ is an 
arbitrary function on $I$,  is the  {\it right composition operator} generated by $f$. We, hereby, denote it by $T_f$. Hence, for each function  
 $g: I \rightarrow \Bbb R$, $T_{f}(g): J \rightarrow {\Bbb R}$ is defined as $T_{f}(g)(\cdot) = g(f(\cdot))$.
\end{defn}

\noindent
As in the case of non-linear operators, we are interested in finding conditions on $f$ in order for $T_{f}$ to map a space of functions $X$ into itself.

\hfill 
\paragraph{Right composition operators on Lipschitz functions and some spaces of functions of bounded variation}
\hfill \break

In \cite{JOS} right $BV$ compositors are completely characterised. 
\begin{defn} (\cite{JOS})
Without loss of generality, take $[a,b] = [0,1]$. For a positive integer $N$, let \[J_{N} = \{X \subseteq [0,1]: X \text{ can be expressed as a union of 
$N$ intervals}\}\] (where the intervals may be open or closed at either end and  singletons are allowed as degenerate closed intervals). Since 
any interval is a union of two subintervals,  $J_{N} \subseteq J_{N+1}$. A function $f: [0,1] \rightarrow {\Bbb R}$ is said to be of $N$-bounded 
variation if $f^{-1}([c, d]) \in J_{N}$ for all $[c, d] \subset {\Bbb R}$. These functions are also called {\it pseudo-monotone} functions (see \cite{ABK}). 
Clearly, every monotone function is pseudo-monotone, indeed it belongs to $BV([0,1])$. Let $BV(N)$ be the set of all functions $f:[0,1] \rightarrow [0,1]$ 
of $N$-bounded variation, and $BV'(N)$ the set of all bounded functions $f:[0,1] \rightarrow {\Bbb R}$ of $N$-bounded variation.
\end{defn}

\begin{rmk}
Clearly, for every $N \in {\Bbb N}$, the following inclusion holds:
\[BV(N) \subseteq BV([0,1]). \]
The inclusion is strict as Example \ref{EXMPJO} shows. 
\end{rmk}

\begin{lem} (\cite{JOS}: Lemma 1) \label{LEMJO}
Every function in $BV'(N)$ is of bounded variation.
\end{lem}

The converse of Lemma \ref{LEMJO} does not hold as the following example shows.

\begin{exmp} (\cite{ABK}: Example 2.1.2) \label{EXMPJO} 
Let 
\[f(x) = \left\{ \begin{array}{ll}
x^{2} {\sin}^{2}(\frac{1}{x})  & \mbox{if } 0< x \leq 1\\ 
0 & \mbox{if } x =0.\\
\end{array}
\right.\]
Then, $f$ is in $BV([0,1])$ since $f'$ exists and is bounded. But $f$ is not pseudo-monotone as $f^{-1}(\{0\})= \{0\} \cup \{\frac{1}{n\pi}: n \in {\Bbb N}\}$.   
\end{exmp}

\begin{thm} (\cite{JOS}: Theorem 3; \cite{ABK}: Theorem 2.1.4)  For $f: [0,1] \rightarrow [0,1]$ , the composition $g \circ f$ belongs to $BV([0,1])$ 
for all $g \in BV'([0,1])$ if and only if $f \in BV'(N)$ for some $N$.  
Moreover, if $f \in BV^{'}(N)$, then $T_{f}$ is bounded. 
\end{thm}
\noindent
All previous results are proved, in \cite{ABK} and \cite{JOS}, on the unit interval $[0,1]$, but, of course, it is the same thing if we work on any interval 
$[a,b]$. In this general setting, we also prove the following results:

\begin{thm} \label{THMLIPR}
The following conditions are equivalent:
\begin{enumerate}
\item[({\it a})]{The function $f: [a,b] \rightarrow [a,b]$ satisfies a Lipschitz condition on $[a,b]$.}
\item[({\it b})]{The operator $T_f$ maps the space $Lip([a,b])$ into itself.}
\end{enumerate}
Moreover, the operator $T_f$ is automatically bounded.
\end{thm}
 \begin{proof}
The implication $(a) \Rightarrow (b)$ and the boundedness of $T_{f}$ follow from the inequality $Lip(g \circ f) \leq  Lip(g) Lip(f)$.\\
The implication $(b) \Rightarrow (a)$ follows from the fact that the identity $f(x) = x$ is Lipschitz continuous. 
 \end{proof}
 
\begin{prop} 
Let $0 < \gamma < 1$. If  the function $f: [a,b]\rightarrow [a,b]$ satisfies a Lipschitz condition on $[a,b]$ then the operator $T_f$ maps the space 
$Lip_{\gamma}([a,b])$ into itself. Moreover, in this case the operator $T_f$ is automatically bounded.
\end{prop}
\begin{proof}
The estimate $Lip_{\gamma}(g \circ f) \leq  Lip_{\gamma}(g) {Lip(f)}^{\gamma}$ implies that the operator $T_f$ maps the space $Lip_{\gamma}([a,b])$ into itself and it is bounded. 
\end{proof}

\begin{prop} 
If the function $f: [a,b]\rightarrow [a,b]$ is absolutely continuous and non-decreasing, then the operator $T_f$ maps the space $AC([a,b])$ into itself.
Moreover, in this case the operator $T_f$ is automatically bounded.
\end{prop}

\begin{proof}
Let $g \in AC([a,b])$ and let $\epsilon >0$. Then, there exists $\delta >0$ such that for all collections $\{ [a_1,b_1],\dots ,[a_n,b_n]\}$ of pairwise 
non-overlapping subintervals of $[a,b]$, the condition \[ \sum_{k=1}^{n} (b_k-a_k) < \delta \] implies that \[ \sum_{k=1}^n \vert g(b_k)-g(a_k) \vert <\epsilon .\] As $f \in AC([a,b])$ 
there exists $\nu >0$ such that, for all collections $\{ [a'_1,b'_1],\dots ,[a'_n,b'_n]\}$ of pairwise non-overlapping subintervals of $[a,b]$, the condition \[ \sum_{k=1}^{n} (b'_k-a'_k) < \nu \] 
implies that \[ \sum_{k=1}^n \vert f(b'_k)-f(a'_k) \vert < \delta \hspace{0.6cm}(\bullet).\] 
We may assume that the intervals $[a_{k}, b_{k}]$ are all non-degenerated and that  $a \leq a'_1< b'_1< \dots < a'_n < b'_n \leq b$ and, as $f$ is non-decreasing, 
we have $a \leq f(a'_1) \leq f(b'_1) \leq \dots \leq f(a'_n) \leq f(b'_n) \leq b$. Hence $\{ [f(a'_1),f(b'_1)],\dots ,[f(a'_n), f(b'_n)]\}$ is a collection of pairwise non-overlapping 
subintervals of $[a,b]$ satisfying condition ($\bullet$), and then it follows that \[ \sum_{k=1}^n \vert g(f(b'_k))-g(f(a'_k)) \vert < \epsilon.\] As $\epsilon$ is arbitrary $g \circ f$ 
is absolutely continuous, that is $g \circ f \in AC([a,b]).$ \\
Next, we prove the continuity (boundedness).  Let $\{g_{n}\}_{n \in {\Bbb N}}$ and $g$ be functions in $AC([a,b])$ with 
\[\lim_{n \rightarrow + \infty} {\Vert g_{n} - g \Vert}_{AC} =0.\]
As 
\begin{align*}
 {\Vert T_{f}(g_{n})  - T_{f}(g) \Vert}_{AC} & =   {\Vert g_{n} \circ f  - g \circ f \Vert}_{BV} \\ 
 & =   \vert (g_{n} \circ f)(a) - (g \circ f)(a) \vert + V((g_{n} -g) \circ f, [a,b])\\
 & \leq   \vert (g_{n} \circ f)(a) - (g \circ f)(a) \vert + V(g_{n} -g, [a,b])\\
 & \leq  \vert (g_{n} \circ f)(a) - (g \circ f)(a) \vert + {\Vert g_{n} - g \Vert}_{AC}\\
 \end{align*}
 and 
 \[\lim_{n \rightarrow + \infty} \vert (g_{n} \circ f)(a) - (g \circ f)(a) \vert = 0,\] 
we have 
\[\lim_{n \rightarrow + \infty} {\Vert T_{f}(g_{n})  - T_{f} (g) \Vert}_{AC}  =0.\]
Hence, the thesis. 
\end{proof}

So, we have a condition, both necessary and sufficient, for the operator $T_f$ to map  $Lip([a,b])$ into itself, and we have shown sufficient conditions 
for the operator $T_f$ to map  $Lip_{\gamma}([a,b])$ and $AC([a,b])$ into themselves. 
\hfill \break

\paragraph{Right composition operators on some spaces of functions}

\begin{defn} (\cite{Z})
A function $f: {\Bbb R} \rightarrow {\Bbb R}$ is called $k-$continuous if for every positive function $\epsilon$ there is a positive function $\delta$ such that 
for any $x,y \in {\Bbb R}$, $\vert x -y \vert < \min \{\delta(x), \delta(y)\}$ implies $\vert f(x) - f(y) \vert  <  \min \{ \epsilon(f(x)) , \epsilon(f(y))\}.$
\end{defn}
 \noindent
Note that every continuous function is $k-$continuous. 

\begin{prop} (\cite{FC3}: Lemma 3.3) \label{propk}
The following properties hold:
\begin{enumerate}
\item{If $f$ and $g$ are $k-$continuous functions, then so is the sum $f+g$.}
\item{If $f$ and $g$ are $k-$continuous functions, then so is the product $fg$.}
\end{enumerate}
\end{prop}

In \cite{KM} and \cite{Z}, right Baire compositors are studied. In \cite{Z}, right Baire one compositors are characterised as follows:
\begin{thm} (\cite{Z}: Theorem 1) \label{theo1}
Let $f: {\Bbb R}  \rightarrow {\Bbb R}$ be a function. Then the following statements are equivalent.
\begin{enumerate}
\item{For any closed subset $A$ of ${\Bbb R}$, $f^{-1}(A)$  is an $F_{\sigma}$ set.}
\item{For any $F_{\sigma}$ set $A$, $f^{-1}(A)$ is an $F_{\sigma}$ set.}
\item{For every positive Baire class one function $\epsilon(\cdot)$, there is a positive function $\delta(\cdot)$ on ${\Bbb R}$ such that 
$\vert x -y \vert < \min \{\delta(x), \delta(y)\}$ implies $\vert f(x) - f(y) \vert  <  \min \{ \epsilon(f(x)) , \epsilon(f(y))$.}
\item{$f$ is a right Baire one compositor. }
\end{enumerate}
\end{thm}

Another characterisation of right Baire one compositors, involving $k-$continuous functions, is given in \cite{FC3}:

\begin{thm} (\cite{FC3}: Theorem 2.6) \label{theo2}
A function $f:{\Bbb R}\rightarrow {\Bbb R}$ is a right Baire one compositor if and only if $f$ is $k-$continuous.
\end{thm}

Next, we give a simple example of a $k$-continuous function that is not continuous.  

\begin{exmp}
The discontinuous function $f = \chi_{\{0\}}$ is $k$-continuous, because $\{0\}$ is both $F_{\sigma}$ and $G_{\delta}$. So $T_{f}$ 
maps ${\mathcal B}_{1}$ into itself, although $T_{f}$
does not map the space of continuous maps into itself.
\end{exmp}

\begin{rmk}  Clearly, since each open set is an $F_{\sigma}$, from Theorem \ref{theo1} it follows that every continuous function is a right Baire one compositor. 
The collection of Baire one compositors lies strictly in between the collection of Baire one functions and the collection of continuous functions.  The 
Riemann function, as Example \ref{EXMPZ} shows, is a Baire one function but it is not a right Baire one compositor. 
\end{rmk}
 
\begin{rmk}
The composition of two right Baire one compositors is a right Baire one compositor. Therefore, if $g: {\Bbb R} \rightarrow {\Bbb R}$ is a right Baire one compositor, 
then, for each $c \in {\Bbb R}$,  $c+ g$ and $cg$ are right Baire one compositors as well because they are the compositions of g and the function h(x) = c + x and k(x) = cx. 
\end{rmk}

However, Zhao writes in \cite{Z} that it is still not clear whether the sum and the product of a continuous function and a right Baire one compositor are right Baire one compositors.
We can give a positive answer to the problem posed by Zhao, combining Proposition \ref{propk} with Theorem \ref{theo2}:

\begin{thm} 
Let $f$ be a continuous function and let $g$ be a right Baire one compositor. Then $f+g$ and $fg$ are right Baire one compositors.
\end{thm}
\begin{proof}
Let $f$ be a continuous function. Then $f$ is also $k-$continuous.
Let $g$ be a right Baire one compositor, then, by Theorem \ref{theo2}, $g$ is a $k-$continuous function. Hence, from the properties in Proposition \ref{propk} 
it follows that $f+g$ and $fg$ are $k-$continuous functions. By Theorem \ref{theo2} this is equivalent to saying that $f+g$ and $fg$ are right Baire one compositors. 
\end{proof}

Next, we give a characterisation of right Baire two compositors.
\begin{thm} \label{theo3}
Let $f: {\Bbb R} \rightarrow {\Bbb R}$. The following conditions are equivalent:
\begin{enumerate}
\item{For any $G_{\delta}$ set $C$, $f^{-1}(C)$ is a $G_{ \delta \sigma}$.} 
\item{For any $G_{\delta \sigma}$ set $C$, $f^{-1}(C)$ is a $G_{ \delta \sigma}$.} 
\item{The operator $T_f$ maps the space ${\mathcal B}_{2}$ into itself (that is, the function $f: {\Bbb R} \rightarrow {\Bbb R}$ is a right Baire two compositor).}
\end{enumerate}
\end{thm}
\begin{proof}
$(1)$ and $(2)$  are clearly equivalent as every $G_{\delta}$ set is a $G_{\delta \sigma}$ set.  \\
$(2) \Rightarrow (3)$ Let $A$ be an open set and let $g \in {\mathcal B}_{2}$. Then $g^{-1}(A)$ is a $G_{\delta \sigma}$. 
Hence, by the hypothesis it follows that $(g \circ f)^{-1}(A) = f^{-1}(g^{-1}(A))$ is a $G_{\delta \sigma}$. Hence, $g \circ f \in {\mathcal B}_{2}$. \\
$(3) \Rightarrow (1)$ Suppose $f$ is a right Baire two compositor and $C$ is a $G_{\delta}$ set. Let $g$ be the characteristic function of $C$. 
As $C$ is both $F_{\sigma \delta}$ and $G_{\delta \sigma}$, $g$ is Baire two. Hence $g \circ f$ is Baire two. Therefore $f^{-1}(C)$ is  a $G_{ \delta \sigma}$ 
set because $f^{-1}(C)= (g \circ f)^{-1}(0,\frac{3}{2}) $.
\end{proof}

The following example shows that there exist functions which are right Baire two compositors but not right Baire one compositors.

\begin{exmp}
It is well-known that $\Bbb Q$ is an $F_{\sigma }$ set, but not a $G_{\delta}$ set. As every $F_{\sigma}$ set is a $G_{\delta \sigma}$ set,  $\Bbb Q$ is 
a $G_{\delta \sigma}$ set. Moreover, ${\Bbb R} \setminus {\Bbb Q}= {\Bbb R} \setminus \cup_{q \in \Bbb Q}\{q\} = \cap_{q \in \Bbb Q} ({\Bbb R} \setminus \{q\}) $ 
is a $G_{\delta}$ set but not an 
$F_{\sigma}$ set. As every  $G_{\delta}$ set is a $G_{\delta \sigma}$ set, ${\Bbb R} \setminus {\Bbb Q}$ is a $G_{\delta \sigma}$ set. Consider the 
function $f=\chi_ {{\Bbb R}\setminus{\Bbb Q}}$. Then $f^{-1}(\{1\})={\Bbb R}\setminus {\Bbb Q}$. As $\{1\}$ is closed, from condition (1) of Theorem \ref{theo1}, 
it follows that $f$ is not a right Baire one compositor.
Now, let $C$ be a $G_{\delta}$ set. Then  \[f^{-1}(C) = \left\{ \begin{array}{ll}
{\Bbb R}  & \mbox{if } 0,1 \in C\\
 {\Bbb R}\setminus {\Bbb Q} & \mbox{if } 1 \in C, 0 \notin C \\
{ \Bbb Q}  & \mbox{if } 0 \in C, 1 \notin C\\
\emptyset  & \mbox{if } 0,1 \notin C\\
\end{array}
\right.\] 
and hence $f^{-1}(C)$ is a $G_{\delta \sigma}$ set. From condition (1) of Theorem \ref{theo3}, it follows that $f$ is a right Baire two compositor.
\end{exmp}

\textbf{Acknowledgment}\\
We express our gratitude to the referee for her/his careful reading of the paper and her/his nice suggestions which were of a great help to improve it.


\begin{thebibliography}{9}
\bibitem{AL0}{A. Alikhani-Koopaei, {\it On the sets of fixed points of bounded Baire one functions}, Asian-Eur. J. Math. {\bf 12} (2019), no. 3, 1950040, 10 pp.}
\bibitem{AL}{A. Alikhani-Koopaei, {\it Equi-Baire one family of functions on metric spaces: a generalization of equi-continuity; and some applications}, Topology Appl. {\bf 277} (2020), 107170, 11 pp.}
\bibitem{AK}{C.D. Aliprantis, K.C. Border, {\it Infinite dimensional analysis. A hitchhiker's guide}, Springer, Berlin, 2006. }
\bibitem{ABM}{J. Appell, J. Bana\'s, N. Merentes, {\it Bounded variations and around}, De Gruyter Series in Nonlinear Analysis and Applications, 17, De Gruyter, Berlin, 2014.}
\bibitem{ABK}{J. Appell, D. Bugajewska, P. Kasprzak, N. Merentes, S. Reinwand, J.L. S\'anchez, {\it Applications of BV type spaces}, Oberwolfach Preprints, 2019.}
\bibitem{AGM}{J. Appell, N. Guanda, N. Merentes, J. L. S\'anchez, {\it Boundedness and continuity properties of nonlinear composition operators: a survey}, Commun. Appl. Anal. {\bf 15} (2011), 153--182.}
\bibitem{AZ}{J. Appell, P.P. Zabrejko,  {\it Nonlinear superposition operators}, Cambridge Tracts in Mathematics {\bf 95}, Cambridge University Press, Cambridge, 1990.}
\bibitem{AZ1}{J. Appell, P.P. Zabrejko,  {\it Remarks on the superposition operator problem in various function spaces}, Complex Var. Elliptic Equ. {\bf 55} (2010), no. 8--10, 727--737.}
\bibitem{BAI}{R. Baire, {\it Sur les fonctions des variables r\'eelles}, Ann. Mat. Pura Appl. {\bf 3} (1899), 1--122.}
\bibitem{BLS1} {G. Bourdaud, M. Lanza de Cristoforis, W. Sickel, {\it Superposition operators and functions of bounded p-variation II}, Nonlinear Anal. {\bf 62} (2005), no. 3, 483--517.}
\bibitem{BLS2}{G. Bourdaud, M. Lanza de Cristoforis, W. Sickel, {\it Superposition operators and functions of bounded p-variation}, Rev. Mat. Iberoam. {\bf 22} (2006), no. 2, 455--487.}
\bibitem{BBT}{A.M. Bruckner, J.B. Bruckner, B.S. Thomson, {\it Real Analysis}, Prentice Hall, United States, 1996}
\bibitem{BBKM}{D. Bugajewska, D. Bugajewski, P. Kasprzak, P. Ma\'ckowiak, {\it Nonautonomous superposition operators in the spaces of functions of bounded variation}, Topol. Methods Nonlinear Anal. {\bf 48} (2016), no. 2, 637--660.}
\bibitem{CDMS2}{D. Candeloro, L. Di Piazza, K. Musia\l,  A.R. Sambucini, {\it Gauge integrals and selections of weakly compact valued multifunctions}, J. Math. Anal. Appl. {\bf 441} (2016), no. 1, 293--308.}
\bibitem{DMS}{L. Di Piazza, V. Marraffa, B. Satco, {\it Closure properties for integral problems driven by regulated functions via convergence results}, J. Math. Anal. Appl. {\bf 466} (2018), no. 1, 690--710.}
\bibitem{FC3}{J.P. Fenecios, E.A. Cabral, {\it $K-$continuous functions and right $B_1$ compositors}, J. Indones. Math. Soc. {\bf 18} (2012), no. 1, 37--44.}
\bibitem{FC1}{J.P. Fenecios, E.A. Cabral, {\it Left Baire-$1$ compositors and continuous functions}. Int. J. Math. Math. Sci. 2013, Art. ID 878253, 3 pp.} 
\bibitem{GS}{M. Goebel, F. Sachweh, {\it On the autonomous Nemytskij operator in H\"older spaces}, Zeitschr. Anal. Anw. {\bf 18} (1999), no. 2, 205--229.}
\bibitem{JOS}{M. Josephy, {\it Composing functions of bounded variation}, Proc. Amer. Math. Soc. {\bf 83} (1981), no. 2, 354-356. }
\bibitem{KM}{O. Karlova, V. Mykhaylyuk, {\it On composition of Baire functions}, Topology Appl. {\bf 216} (2017), 8--24.}
\bibitem{KE}{S. Kechris, {\it Classical Descriptive Set Theory}, Springer-Verlag, New York, 1995.}
\bibitem{KU}{K. Kuratowski, {\it Topology, Volume I}, Academic Press, London, 1966.}
\bibitem{LP}{M. Laczkovich, D. Preiss, {\it $\alpha$-variation and transformation into $C^n$ functions}, Indiana Univ. Math. J. {\bf 34} (1985), no. 2, 405--424.}
\bibitem{LEC}{D. Lecomte, {\it How we can recover Baire class one functions?}, Mathematika {\bf 50} (2003), no. 1--2, 171--198.}
\bibitem{LTZ}{ P.-Y. Lee, W.-K.Tang, D. Zhao, {\it An equivalent definition of functions of the first Baire class}, Proc. Amer. Math. Soc. {\bf 129} (2001), no. 8, 2273--2275.} 
\bibitem{M}{P. Ma\'ckowiak, {\it On the continuity of superposition operators in the space of functions of bounded variation}, Aequationes Math. {\bf 91} (2017), no. 4, 759--777.}
\bibitem{MM}{M. Marcus, V.J. Mizel, {\it Complete characterization of functions which act, via superposition, on Sobolev spaces}, Trans. Amer. Math. Soc. {\bf 251} (1979), 187-218.}
\bibitem{NAT}{I.P. Natanson, {\it Theory of Functions of the Real Variable, Volume II}, Dover Publications, New York, 2016.}
\bibitem{SA}{B.-R. Satco, {\it Nonlinear Volterra integral equations in Henstock integrability setting.} Electron. J. Differential Equations 2008, no. 39, 9pp.}
\bibitem{Z}{D. Zhao, {\it Functions whose composition with Baire class one functions are Baire class one}, Soochow J. Math. {\bf 33} (2007), no. 4, 543--551.}
\end{thebibliography}
\end{document}